\documentclass[letter,reqno,11pt]{amsart}
\usepackage{amsmath,color,verbatim,enumerate,amssymb,amsthm,thmtools,graphics,marvosym,amsfonts,algorithm,caption,float,multirow,array,booktabs,tikz-cd,xcolor}

\usepackage[noend]{algpseudocode}
\usepackage[utf8]{inputenc}
\usepackage{hyperref}
\usepackage[nameinlink]{cleveref}
\hypersetup{	colorlinks,	linkcolor={blue},	citecolor={blue},	urlcolor={blue}}
\usepackage{multirow}
\allowdisplaybreaks
\definecolor{ao(english)}{rgb}{0.0, 0.5, 0.0}
\newtheorem{theorem}{Theorem}
\newtheorem{lemma}{Lemma}
\newtheorem{proposition}{Proposition}
\newtheorem{corollary}{Corollary}
\theoremstyle{definition}
\newtheorem{definition}{Definition}
\declaretheorem[style=definition]{example}
\newtheorem{remark}{Remark}

\DeclareMathOperator{\lcm}{lcm}
\newcommand{\fq}{{\mathbb F}_{q}}
\newcommand{\pq}{\mathbb{P}^1(\fq)}

\newcommand{\pdiv}{\mid\mspace{-1mu}\mid}
\begin{document}

\title{R\'edei permutations with cycles of the same length}
\author{Juliane Capaverde, Ariane M. Masuda,  and Virg\'inia M. Rodrigues}
\date{\today}
\address{Departamento de Matemática Pura e Aplicada, Universidade Federal do Rio Grande do Sul, Avenida Bento Gon\c{c}alves, 9500, Porto Alegre, RS 91509-900 Brazil}
\email{juliane.capaverde@ufrgs.br, vrodrig@mat.ufrgs.br}
\address{Department of Mathematics, New York City College of Technology, CUNY, 300 Jay Street, Brooklyn, NY 11201 USA}
\email{amasuda@citytech.cuny.edu}
\thanks{The second author received support for this project
provided by a PSC-CUNY grant, jointly funded by The Professional Staff Congress and The City University of New York.}
\keywords{R\'edei function, involution, permutation, cycle decomposition}

\maketitle

\begin{abstract}
Let $\fq$ be a finite field of odd characteristic. We study R\'edei functions that induce permutations over $\pq$ whose cycle decomposition contains only cycles of length $1$ and $j$, for an integer $j\geq 2$. When $j$ is $4$ or a prime number, we give necessary and sufficient conditions for a Rédei permutation of this type to exist over $\pq$, characterize R\'edei permutations consisting of $1$- and $j$-cycles, and determine their total number. We also present explicit formulas for R\'edei involutions based on the number of fixed points, and procedures to construct R\'edei permutations with a prescribed number of fixed points and $j$-cycles for $j \in \{3,4,5\}$.

%\keywords{R\'edei function \and involution \and permutation \and cycle decomposition}
\end{abstract}

%\section{Introduction template}
%\label{intro}
%Your text comes here. Separate text sections with
%\section{Section title}
%\label{sec:1}
%Text with citations %\cite{RefB} and \cite{RefJ}.
%\subsection{Subsection title}
%\label{sec:2}
%as required. Don't forget to give each section
%and subsection a unique label (see Sect.~\ref{sec:1}).
%\paragraph{Paragraph headings} Use paragraph headings as needed.
%\begin{equation}
%\end{equation}

%\begin{figure}
%  \includegraphics{example.eps}
%\end{figure}
%\begin{figure*}
%  \includegraphics[width=0.75\textwidth]{example.eps}
%\end{figure*}
%\begin{table}
%\begin{tabular}{lll}
%\hline\noalign{\smallskip}
%\end{tabular}
%\end{table}

\section{Introduction}

Permutations over finite fields have been extensively studied over the past decades. They play a crucial role in communication theory, where it is often advantageous to have permutations that are easily implemented and require a reasonably small amount of storage space. They are used in coding theory, cryptography, combinatorial design, among others. In particular, the cycle decomposition and the number of fixed points of permutations are related to properties of codes and cryptographic schemes where they are employed.

Several families of functions over finite fields have been studied in the literature with the goal of determining which elements in these families are permutations, as well as describing their cycle structure~\cite{AHMAD1969370,MR2435050,MR1159877,MR1542258}. One of these families consists of the so-called \textit{R\'edei functions}. Let $\fq$ be a finite field of odd characteristic with $q$ elements, and let $\mathbb P^1(\fq):=\mathbb F_q\cup \{\infty\}$. Consider the binomial expansion 
$\displaystyle (x+\sqrt y)^n = N(x,y)+D(x,y)\sqrt{y}.$
%$$(x+\sqrt y)^n = N(x,y)+D(x,y)\sqrt{y}.$$  
For $n\in\mathbb N$ and $a \in \fq$, the R\'edei function $R_{n,a}\colon \mathbb P^1(\mathbb F_q)  \to \mathbb P^1(\mathbb F_q)$ is defined by
\begin{align*}
R_{n,a}(x)=
\begin{cases} \dfrac{N(x,a)}{D(x,a)} & \text{ if } D(x,a)\neq 0, x\neq\infty
 \\
\infty & \text{ otherwise.}
\end{cases}
\end{align*}
R\'edei functions have been used in several applications such as pseudorandom number generation~\cite{MR2401984}, cryptography~\cite{MR3475548,key:article,MR3784184,MR815132}, coding theory~\cite{MR2966683}, solving Pell equations~\cite{MR2766784}, and the construction of other classes of permutation polynomials~\cite{MR3954416}. We are interested in R\'edei functions that are bijections, called \textit{R\'edei permutations}. It is well-known~\cite{MR3384830} that $R_{n,a}$ induces a permutation on $\mathbb P^1(\fq)$ if and only if $\gcd(n,q-\chi(a))=1$, where $\chi(a)$ is the quadratic character of $a$, that is, $\chi(a) = 1$ if $a$ is a square in $\fq$ and $\chi(a)=-1$ otherwise. In particular, this implies that $R_{n,a}$ is not a permutation for any even integer $n$. A R\'edei permutation $R_{n,a}$ consists of $\varphi(d)/o_d(n)$ disjoint $o_d(n)$-cycles for each $d\mid q-\chi(a)$, in addition to $\chi(a)+1$ cycles of length $1$, where $o_d(n)$ denotes the order of $n$ modulo $d$ and $\varphi$ is Euler's totient function~\cite{MR3384830}. Cycles of length $1$ correspond to fixed points. We note that $0$ and $\infty$ are always fixed points  of Rédei permutations, and when $\chi(a)=1$, the points $\sqrt{a}$ and $-\sqrt{a}$ are also fixed. The number of fixed points of $R_{n,a}$ is $\gcd(n-1,q-\chi(a)) + \chi(a)+1$; see~\cite{MR3384830}. An explicit formula for the fixed points of a R\'edei function is given in~\cite{Chubb2019}. We observe that the cycle structure of a R\'edei permutation does not depend on the field element $a$, but only on $\chi(a)$ and $n$. More precisely, if $R_{n,a}$ and $R_{n,b}$ are permutations with $\chi(a) = \chi(b)$, then they are distinct functions over $\mathbb P^1(\fq)$ with the same cycle decomposition.

In this paper we study R\'edei permutations that decompose into cycles of the same length. This type of function is of interest in the construction of interleavers for turbo codes~\cite{MR2092636,MR2966683,Sakzad2010}. Characterizations of permutations with cycles of the same length have been given for monomial permutations~\cite{MR2092636}, Dickson polynomials~\cite{MR2436339}, R\'edei and M\"obius functions~\cite{MR2966683,Sakzad2010}, and linear maps~\cite{MR3911214}. The results in~\cite{MR3384830,MR2966683,Sakzad2010} reveal a strong connection between R\'edei and monomial functions. In fact, the description of R\'edei permutations with cycles of the same length given in~\cite{MR2966683,Sakzad2010} is similar to that given for monomial permutations in~\cite{MR2092636}. However, it is not clear how to produce permutations with all nontrivial cycles of a given length $j$.  

Among the permutations with all nontrivial cycles of the same length, those with $2$-cycles are specially important. They are called \textit{involutions} and have the property of being their own inverses. Since in most applications both the permutation and its inverse must be stored in memory, involutions are desirable in environments with limited resources. In particular, they can be applied to produce self-inverse interleavers for turbo codes, allowing the same structure and technology used for encoding to be used for decoding as well~\cite{Sakzad2010}. In~\cite{MR3587256,MR3480112} the authors consider involutions over finite fields of characteristic two. They study several classes of polynomials such as monomials, linear maps and Dickson polynomials of the first kind with the goal of characterizing involutions in each class and providing explicit constructions. They also present results on the number of fixed points of these involutions, pointing out that for cryptographic applications such as the design of S-boxes, it is also desirable to have permutations with a small number of fixed points. Other studies on involutions have since then appeared for fields of any characteristic. In~\cite{MR3651299}  explicit formulas for $n$ such that $x^n$ is an involution with a prescribed number of fixed points are provided. In~\cite{MR4071834,MR4038906} other constructions of involutions are given. Motivated by these works, we further investigate R\'edei functions with short cycle length and their number of fixed points, including  involutions. 

The composition of R\'edei functions is also a R\'edei function, namely, $R_{m,a}\circ R_{n,a} = R_{mn,a}$. It follows that the number of fixed points in the $j^{\mathrm{th}}$ iteration of $R_{n,a}$ is $\gcd(n^j-1,q-\chi(a))+\chi(a)+1$. This means that a necessary condition for $R_{n,a}$ to decompose into $1$- and $j$-cycles is that $\gcd(n^j-1, q-\chi(a)) = q-\chi(a)$. If $j$ is prime, then this condition is also sufficient. In fact, the equality $\gcd(n^j-1, q-\chi(a)) = q-\chi(a)$ implies that all elements are fixed in the $j^{\mathrm{th}}$ iteration; moreover, if an element belongs to a cycle of length $k$ and is fixed in the $j^{\mathrm{th}}$ iteration, then $k$ must be a divisor of $j$. Furthermore, if $R_{n,a}$ has $d+\chi(a)+1$ fixed points, then as we mentioned before $d=\gcd(n-1,q-\chi(a))$, and therefore $d$ is a common divisor of $n-1$, $n^j-1$ and $q-\chi(a)$. When $j$ is prime, we are able to determine exactly which divisors $d$ of $q-\chi(a)$ have the property that the integer $d+\chi(a)+1$ is the number of fixed points of a R\'edei permutation with cycles of the same length $j$. We call these divisors $(q,\chi(a),j)$-\textit{admissible}. The questions we seek to answer can be translated in terms of $(q,\chi,j)$-admissible integers for $\chi \in\{-1,1\}$. For instance, the existence of a R\'edei permutation over $\pq$ with $1$- and $j$-cycles is equivalent to the existence of a $(q,\chi,j)$-admissible divisor of $q-\chi$ with $\chi\in\{-1,1\}$.

We focus on R\'edei permutations with $1$- and $p$-cycles, where $p$ is a prime number. In Section~\ref{cycleslengthp} we give a characterization of the integers $n$ for which $R_{n,a}$ is of this type. When $p\neq 2$, we give necessary and sufficient conditions for an integer $d$ to be $(q,\chi,p)$-admissible, and  determine the number of R\'edei permutations with $1$- and $p$-cycles. In Section~\ref{sec_involutions} we consider the case $p=2$. Besides the characterization of $(q,\chi,2)$-admissible integers, we explicitly find all R\'edei involutions and identify those ones with two and four fixed points (the smallest possible amounts of fixed points of a general R\'edei function). We also characterize and give a procedure to construct R\'edei permutations with $1$- and $4$-cycles. In Section~\ref{cycleslength3} we apply the results from Section~\ref{cycleslengthp} to $p=3$ and $5$. In each case we provide a procedure to construct R\'edei permutations with $1$- and $p$-cycles. Finally, in Section~\ref{examples} we present examples that illustrate the procedures outlined in previous sections.

We note that our results can be applied to monomial permutations as the cycle decomposition of $x^n$ over $\fq$ is the cycle decomposition of $R_{n,a}$ over $\pq$ with $\chi(a)=1$ and one less cycle of length $1$.

\section{Cycles of length $1$ and $p$}\label{cycleslengthp}

Throughout the text, $q$ is an odd prime power and $\chi \in \{-1,1\}$. We study Rédei functions that induce permutations over $\pq$ whose cycle decomposition contains only cycles of length $1$ and $j$ with $j\geq 2$. We refer to cycles of length $j$ as $j$-cycles. In addition, whenever we say that a permutation has $1$- and $j$-cycles, we mean that there is no cycle of other length in the decomposition of the permutation.

\begin{definition}
	Let $2\leq j < q-\chi$. An integer $d$ is \emph{$(q,\chi,j)$-admissible} if there is a R\'edei permutation $R_{n,a}$ over $\mathbb P^1(\fq)$ with $\chi(a)=\chi$ that decomposes into $d+\chi+1$ fixed points and $j$-cycles.
\end{definition}

A R\'edei function $R_{n,a}$ has $\gcd(n-1,q-\chi(a))+\chi(a)+1$ fixed points, so a $(q,\chi,j)$-admissible integer is a  divisor of $q-\chi$. This divisor cannot be $q-\chi$ as otherwise  there would be no $j$-cycle with $j\geq 2$. Moreover, since  $\gcd(n,q-\chi(a)) = 1$, we must have that $n$ is odd, so a $(q,\chi,j)$-admissible integer is always even.

In this section our focus is on the case when $j=p$ is prime. We recall that  $R_{n,a}$ decomposes into $1$- and $p$-cycles if and only if $q-\chi(a)\mid n^p-1$. The next lemma will be used to completely describe the $(q, \chi, p)$-admissible integers in Propositions~\ref{p-admissible} and ~\ref{2-adm}. These descriptions are based on the prime factorization of $q-\chi$.

\begin{lemma}\label{gcdpcycles}
	Let $p$ be a prime and $n$, $k$ be positive integers such that $n$ is odd and $k\mid n^p-1$. Let $k= p_1^{\alpha_{1}}\cdots p_r^{\alpha_{r}}$ be the prime factorization of $k$. Then
	$$\gcd(n-1,k) =  p_1^{\beta_{1}}\cdots p_r^{\beta_{r}} \quad\text{  and  }\quad \gcd(n^{p-1}+n^{p-2}+\cdots +1,k)=  p_1^{\gamma_{1}}\cdots p_r^{\gamma_{r}},$$ 
	where 
	$$
	\begin{cases}
	\text{ either } (\beta_i,\gamma_i) \text{ or } (\gamma_i,\beta_i)\in\{(0,\alpha_i)\} & \text{ if } p_i\neq p \\
	\beta_i=\gamma_i= 1  & \text{ if } p_i = p \text{ and } \alpha_i=1 \\
	(\beta_i,\gamma_i) \text{ or } (\gamma_i,\beta_i)\in\{(1,2)\} & \text{ if } p_i= p=2 \text{ and } \alpha_i =  2 \\
	\text{ either } (\beta_i,\gamma_i) \text{ or } (\gamma_i,\beta_i)\in\{1\} \times\{\alpha_i-1,\alpha_i\} & \text{ if } p_i= p=2 \text{ and } \alpha_i > 2 \\
	(\beta_i,\gamma_i) \in\{\alpha_i-1,\alpha_i\}\times \{1\} & \text{ if } p_i=p>2 \text{ and } \alpha_i\geq 2.
         \end{cases}
         $$ 
\end{lemma}

\begin{proof}
	Since $k \mid n^p-1$, we have $ p_1^{\alpha_1}\cdots p_r^{\alpha_r} \mid (n-1) (n^{p-1} + n^{p-2} + \cdots +1)$. If $p_i$ divides both $n-1$ and $n^{p-1} + n^{p-2}+ \cdots +1$ for some $1\leq i \leq r$, then $n\equiv 1\pmod{p_i}$ and $0\equiv n^{p-1} + n^{p-2}+\cdots +1 \equiv p \pmod{p_i}$. Thus $p_i=p$. As a consequence, if $p_i \neq p$, then $p_i^{\alpha_i}$ divides either $n-1$ or $n^{p-1} + n^{p-2}+\cdots +1$. 
	
	It turns out that $p \mid n^{p-1} + n^{p-2}+\cdots +1$ if and only if $p \mid n-1$. 
	Suppose that $p_i=p=2$ and $\alpha_i \geq 2$. The integers  $(n-1)/2$ and $(n+1)/2$ are consecutive,  so they are coprime with different parity. Hence either $\beta_i =1$ or $\gamma_i =1$. Furthermore, if $\gamma_i=1$, then $(n+1)/2$ and $n-1$ are coprime and we have 
		\begin{align*}
		k = \gcd(n^2-1, k) & =  2 \cdot\gcd \left( (n+1)(n-1)/2, k/2\right) \nonumber \\
		& = 2\cdot \gcd \left((n+1)/2, k/2\right)\cdot \gcd \left( n-1,k/2\right) \nonumber \\
		& =  \gcd( n+1, k)\cdot \gcd \left( n-1,k/2\right).
	\end{align*}
	Similarly, when $\beta_i=1$, we have
	\begin{equation*}
	k= \gcd(n-1,k) \cdot \gcd(n+1,k/2).
	\end{equation*}
	By comparing the powers of 2 in the above expressions, we obtain the desired result when $p_i=p=2$.
	
	Suppose that $p_i=p$ is odd and $\alpha_i \geq 2$. We write $n=p\ell +1$ to compute
	\begin{align*} n^{p-1} + n^{p-2}+\cdots +1 & = \sum_{i=0}^{p-1} \sum_{j=0}^i \binom{i}{j}p^j\ell^j \\ & \equiv p + p\ell\sum_{i=1}^{p-1} i \pmod{p^2} \\ & \equiv p+ \dfrac{\ell
	 (p-1)}{2}p^2\equiv p \pmod{p^2}.\end{align*}
	This means that $p^2 \nmid n^{p-1} + n^{p-2}+\cdots +1$. Moreover, $(n-1)/p$ and $(n^{p-1} + n^{p-2}+\cdots +1)/p$ are coprime. Therefore $\gamma_i =1$ and $\beta_i \in\{\alpha_i -1,\alpha_i\}$. \end{proof}

When $k=q-\chi$,  we have the following direct consequence of Lemma~\ref{gcdpcycles}. 

\begin{lemma}\label{pcycle_inv_number_fp}
	Let $p$ be a prime, $q-\chi = p_1^{\alpha_{1}}\cdots p_r^{\alpha_{r}}$ be the prime factorization of $q-\chi$, and $d$ be $(q,\chi,p)$-admissible.
	We write $d = p_1^{\beta_{1}}\cdots p_r^{\beta_{r}}$ with $0\leq \beta_i \leq \alpha_i$.  Then
	\begin{equation}\label{betai}
	    \beta_i = \begin{cases}
	    0\text{ or } \alpha_i&  \text{ if } p_i \neq p \\
	    1 & \text{ if } p_i=p\  \text{ and } \alpha_i=1 \\ 
	    1,\alpha_i-1\text{ or }  \alpha_i&  \text{ if } p_i=p=2\text{ and } \alpha_i\geq 2\\
	     \alpha_i-1\ or\ \alpha_i &  \text{ if } p_i=p > 2 \text{ and } \alpha_i\geq 2.  
		\end{cases}
		\end{equation} 
\end{lemma}

The above result gives necessary conditions for a divisor of $q-\chi$ to be $(q,\chi,p)$-admissible. In particular,  a $(q,\chi,p)$-admissible integer $d$ satisfies $\gcd(d, (q-\chi)/d)=1$ or $p$. 

We can now characterize the integers $n$ for which there is a R\'edei permutation  $R_{n,a}$ with a prescribed number of fixed points and $p$-cycles. We denote the $p$-adic valuation of an integer $z$ by  $\nu_{p}(z)$, that is, $\nu_{p}(z) = \max \{ \alpha \in \mathbb{N}\colon p^{\alpha}\mid z\}$.

\begin{theorem}\label{pcycle_inv_syst}
	Let $p$ be a prime, $q-\chi = p_1^{\alpha_{1}}\cdots p_r^{\alpha_{r}}$ be the prime factorization of $q-\chi$, and $d= p_1^{\beta_{1}}\cdots p_r^{\beta_{r}}$ be a proper divisor of $q- \chi$  satisfying~\eqref{betai}. The R\'edei permutation $R_{n,a}$ has $d  +\chi+1$ fixed points and $p$-cycles if and only if the following conditions are satisfied:
	\begin{enumerate}[\normalfont(i)]
		\item $n$ is a solution to
		\begin{equation} \label{pcycle_syst}
		\begin{cases}
		x \equiv  1 & \pmod {d} \\
		x^{p-1}+x^{p-2}+\cdots+x +1\equiv 0 &\pmod {(q-\chi)/d},
		\end{cases}
		\end{equation}
		\item if $\nu_p(d) = \nu_p(q-\chi)-1$, then $\nu_p(n-1) = \nu_p(d)$.
	\end{enumerate}
\end{theorem}

\begin{proof}
	First, assume that $R_{n,a}$ has  $d + \chi+1$ fixed points and $p$-cycles. Then $q-\chi\mid n^p-1$ and  $\gcd(n-1,q-\chi)=d$. Thus (i) and (ii) hold.
	 
	  Conversely, it follows  that $q-\chi\mid n^p-1$. We also have $d \mid \gcd(n-1,q-\chi)$ and $(q-\chi)/d \mid \gcd(n^{p-1} + \cdots +1,q-\chi)$. Suppose $\gcd(n-1,q-\chi) = p_1^{\theta_1}\cdots p_r^{\theta_r}$. We want to show that $\theta_i=\beta_i$ for $1\leq i \leq r$. It suffices to consider the case $\beta_i < \alpha_i$. By~\eqref{betai}, we have three possibilities.
	 \begin{enumerate}[1)]
	     \item $p_i\neq p$ and $\beta_i=0$ \\
	     Then $p_i$ divides $(q-\chi)/d$ and consequently divides $\gcd(n^{p-1} + \cdots +1,q-\chi)$. By Lemma~\ref{gcdpcycles}, $\theta_i=0$. 
	     \item $p_i=p=2$, $\alpha_i \geq 3$ and $\beta_i=1$ \\ Then $4$ divides both $(q-\chi)/d$ and $\gcd(n^{p-1} + \cdots +1,q-\chi)$. By Lemma~\ref{gcdpcycles}, $\theta_i=1$.
	     \item $p_i=p$, $\alpha_i \geq 2$ and $\beta_i = \alpha_i-1$ \\
	     In this case Condition (ii) implies that $p_i^{\alpha_i} \nmid n-1$, thus $\theta_i=\beta_i$. 
	 \end{enumerate} \end{proof}

It is a known fact that $R_{m,a} = R_{n,a}$ if and only if $m \equiv n \pmod{q-\chi}$. Thus we are interested in the solutions to~\eqref{pcycle_syst} modulo $q-\chi$.  In particular, we count the number of solutions that are different modulo $q-\chi$, unless otherwise stated.

In Section~\ref{sec_involutions} we prove that the necessary conditions given in Lemma~\ref{pcycle_inv_number_fp} for an integer to be $(q,\chi,p)$-admissible are actually sufficient when $p=2$. In the remaining part of this section we assume that $p$ is an odd prime. Our goal is to completely characterize $(q,\chi,p)$-admissible integers.
 
 By Theorem~\ref{pcycle_inv_syst}, in order for a divisor $d$ of $q-\chi$ satisfying the necessary conditions  in Lemma~\ref{pcycle_inv_number_fp} to be $(q,\chi,p)$-admissible, there must exist a solution to  $x^{p-1} + x^{p-2} + \cdots + x +1\equiv 0 \pmod {(q-\chi)/d}$. We make use of the following known result concerning the cyclotomic polynomial $\phi_p(x) = x^{p-1}+\cdots+x+1$. 
 
\begin{lemma}[{\cite[Theorem 2.1.125]{MR3087321}}]\label{cyclotomic_roots}
	Let $p$ and $p'$ be primes. If $p\neq p'$, then $\phi_p$ factors into the product of $(p-1)/o_{p}(p')$ distinct irreducible polynomials of degree $o_{p}(p')$ in $\mathbb{F}_{p'}[x]$. If $p=p'$, then $\phi_p(x) = (x-1)^{p-1}$ in $\mathbb{F}_{p'}[x]$.
\end{lemma}

Therefore, for $p\neq p'$, the equation $\phi_p(x) \equiv 0 \pmod{p'}$ has a solution if and only if $p' \equiv 1 \pmod p$, and in this case there are exactly $p-1$ distinct solutions modulo $p'$. If $p=p'$, then the only solution is $n\equiv 1 \pmod{p'}$. Each solution modulo $p'$ can be lifted to a unique solution modulo $(p')^{\ell}$ by Hensel's Lemma. Hence the existence of a solution to $\phi_p(x) \equiv 0 \pmod {(q-\chi)/d} $ is guaranteed, if the factorization of $(q-\chi)/d$ contains only primes of the form $pk+1$ or $p$ itself. 

\begin{proposition}\label{p-admissible}
	Let $p$ be an odd prime, $q-\chi = p_1^{\alpha_1} \cdots p_r^{\alpha_r}$ be the prime factorization of $q-\chi$ and $d = p_1^{\beta_1} \cdots p_r^{\beta_r}$ be a proper divisor of $q - \chi$. Then $d$ is $(q,\chi,p)$-admissible if and only if 
		\begin{equation}\label{p-adm-cond} \beta_i = \begin{cases}
		\alpha_i-1\text{ or } \alpha_i & \text{ if } p_i=p\text{ and } \alpha_i\geq 2 \\
		0\text{ or } \alpha_i & \text{ if }  p_i \equiv 1 \pmod p \\
		\alpha_i & \text{ otherwise}.
		\end{cases}\end{equation}
\end{proposition}

\begin{proof}
	Suppose that $d$ is $(q,\chi,p)$-admissible. By Lemma~\ref{pcycle_inv_number_fp}, if $p_i=p$, then $\beta_i$ is as claimed. If $p_i\neq p$, then $\beta_i \in \{0,\alpha_i\}$. Let $R_{n,a}$ be a R\'edei permutation with $d+\chi+1$ fixed points and $p$-cycles. By Theorem~\ref{pcycle_inv_syst}, we have  $\phi_p(n) \equiv 0 \pmod{(q-\chi)/d}$, and  by Lemma~\ref{cyclotomic_roots} as discussed above, if $p_i \not\equiv 1 \pmod p$, then $\beta_i = \alpha_i$.
	
	For the converse, assume that each $\beta_i$ satisfies~\eqref{p-adm-cond}. Then there exists an integer $k$ such that $\phi_p(k)\equiv 0 \pmod{(q-\chi)/d}$. Thus~\eqref{pcycle_syst} has a solution if and only if  
	\begin{equation}\label{plinear_congruences}
	\begin{cases}
	x\equiv 1 &\pmod d \\ 
	x \equiv k & \pmod{(q-\chi)/d}
	\end{cases}
	\end{equation}
has a solution, or equivalently,  $\gcd(d,(q-\chi)/d)$ divides $k-1$. By~\eqref{p-adm-cond}, this $\gcd$ is either $1$ or $p$. If it is 1, there is nothing to prove. Otherwise, $p\mid (q-\chi)/d$, which implies that $\phi_p(k)\equiv 0 \pmod p$, that is, $k \equiv 1 \pmod p$. Hence $p\mid k-1$ as we wanted. We conclude that~\eqref{plinear_congruences} has a solution. Furthermore, the solution is unique modulo 
	\[\lcm\left(d,(q-\chi)/d\right) = \begin{cases}
	q-\chi & \text{ if }  \nu_{p}(d) = \nu_p(q-\chi)\\ 
	(q-\chi)/p & \text{ if } \nu_{p}(d) = \nu_p(q-\chi) -1.
	\end{cases}\]
When $\nu_p(d)= \nu_{p}(q-\chi)-1$,  there are $p$ solutions to~\eqref{plinear_congruences} modulo $q-\chi$, and it follows that exactly $p-1$ of them satisfy  Theorem~\ref{pcycle_inv_syst}(ii). In any case, there is at least one integer $n$ such that $R_{n,a}$ is a permutation with the required properties. 	
  \end{proof}

Proposition~\ref{p-admissible} can be used to obtain an existence condition  on R\'edei permutations with $1$- and $p$-cycles over $\pq$. 

\begin{theorem}\label{p-cycle-existence}
	Let $p$ be an odd prime. There exists a R\'edei permutation over $\pq$ with $1$- and $p$-cycles if and only if $q-1$ or $q+1$ has a prime factor of the form $pk+1$ or is divisible by $p^2$.
\end{theorem}

We can determine the total number of such R\'edei permutations with a prescribed number of fixed points for a fixed parameter $a$ with quadratic character $\chi$. 

\begin{proposition}\label{pcycle_number}
Let $p$ be an odd prime and $d$ be $(q,\chi,p)$-admissible. Let $M_d$ be the number of R\'edei permutations $R_{n,a}$ with $d+\chi+1$ fixed points, $p$-cycles, and a fixed parameter $a$ with $\chi(a)=\chi$. Then 
	\[ M_d= \begin{cases} (p-1)^{u} & \text{ if } \nu_{p}(d) = \nu_{p}(q-\chi) \\ 
	(p-1)^{u+1}& \text{ if } \nu_{p}(d) = \nu_{p}(q-\chi) -1,\end{cases}\]
	where $u= |\{p'\ prime\colon p'\equiv 1 \pmod p, p'\mid q-\chi \ and\ p'\nmid d\}|$.
\end{proposition}
	
\begin{proof}
Suppose $(q-\chi)/d = p^{\alpha_0}p_1^{\alpha_1}\cdots p_u^{\alpha_u}$. For each prime divisor $p_i \neq p$ of $(q-\chi)/d$, the congruence $x^{p-1}+\cdots +x+1 \equiv 0 \pmod{p_i^{\alpha_i}}$ has $p-1$ solutions (distinct modulo $p_i^{\alpha_i}$).  There is one solution modulo $p^{\alpha_0}$, namely, $x\equiv 1 \pmod p$. By applying the Chinese Remainder Theorem, we obtain $(p-1)^u$ solutions to $\phi_p(x) \equiv 0 \pmod {(q-\chi)/d}$. 
		
For each integer $k$ such that $\phi_p(k)\equiv 0 \pmod{(q-\chi)/d}$, as pointed out in the proof of  Proposition~\ref{p-admissible}, there is exactly one corresponding solution to~\eqref{pcycle_syst} if $\nu_{p}(d) = \nu_{p}(q-\chi)$, and $p-1$ solutions to~\eqref{pcycle_syst} satisfying  Theorem~\ref{pcycle_inv_syst}(ii) if $\nu_p(d) = \nu_p(q-\chi)-1$. 
  \end{proof}

We observe that for each $(q,\chi,p)$-admissible $d$ the cycle structures of all  permutations in Proposition~\ref{pcycle_number} are the same.
	
\begin{proposition}\label{pcycle_number_total}
Let $p$ be an odd prime and $M$ be the number of R\'edei permutations $R_{n,a}$ with $1$- and $p$-cycles and a fixed parameter $a$ with $\chi(a)=\chi$. Then
\[ M= \begin{cases} p^{r}-1 & \text{ if } p^2\nmid q-\chi \\ p^{r+1}-1 & \text{ if } p^2\mid q-\chi,\end{cases}  \]
where $r= |\{p'\ prime\colon p'\equiv 1 \pmod p, p'\mid q-\chi\}|$.
\end{proposition}
	
\begin{proof}
If $p^2\nmid q-\chi$, then $\nu_{p}(d) = \nu_{p}(q-\chi)=1$ for any $(q,\chi,p)$-admissible $d$, and there is one such $d$ for each choice of $u$ prime factors of $q-\chi$ of the form $pk+1$. Note that if $d=q-\chi$,  the permutation is the identity map. Then
\[ M= \sum_{u=0}^r \binom{r}{u} (p-1)^{r-u} -1= p^r -1. \] 
		
If $p^2\mid q-\chi$, then $\nu_{p}(d) = \nu_{p}(q-\chi)$ or $\nu_{p}(d) = \nu_{p}(q-\chi)-1$. For each choice of $u$ prime factors of $q-\chi$ of the form $pk+1$, there are two $(q,\chi,p)$-admissible divisors $d$. Hence
		\begin{align*}
		 M &= \sum_{u=0}^r \binom{r}{u} (p-1)^{r-u} + \sum_{u=0}^r \binom{r}{u} (p-1)^{r-u+1} -1\\
		 &=  p^r +(p-1)p^r -1= p^{r+1} -1. 
		 \end{align*} 
  \end{proof}

Given a finite field $\mathbb F_q$, we describe a procedure that determines all odd primes $p$ for which there is a R\'edei permutation  over ${\mathbb{P}^1}(\mathbb F_{q})$ with $1$- and $p$-cycles along with the total number of each cycle type. 

\begin{enumerate}[1.]
\item For each $\chi\in\{-1,1\}$ and odd prime $p$, do the following.
\begin{enumerate}
	\item Apply Theorem~\ref{p-cycle-existence}  to check whether or not there exists an $R_{n,a}$ with $\chi(a)=\chi$ and $1$- and $p$-cycles. If it exists, Proposition~\ref{pcycle_number_total} gives  the total number $M$ for each $a\in\mathbb F_q$ with $\chi(a)=\chi$.
	\item Apply Proposition~\ref{p-admissible} to find the admissible integers $d$.
	\item For each $d$,
		\begin{enumerate}
		\item use Proposition~\ref{pcycle_number} to find $M_d$,
		\item the number of fixed points is $d+\chi+1$,
		\item the number of $p$-cycles is $(q-d-\chi)/p$. 
		\end{enumerate}
\end{enumerate}
\end{enumerate}

\section{Involutions}\label{sec_involutions}

We begin by characterizing the $(q,\chi,2)$-admissible integers. In this case, the necessary conditions given in Lemma \ref{pcycle_inv_number_fp} for an integer to be $(q,\chi,2)$-admissible mean that $d$ must be an even divisor of $q-\chi$ such that $\nu_{2}(d) \in \{1, \nu_2(q-\chi)-1, \nu_2(q-\chi)\}$ and $\gcd(d,(q-\chi)/d)\mid 2$. When $p=2$, the system of congruences~\eqref{pcycle_syst} takes the form
\begin{equation} \label{syst}
\begin{cases}
x \equiv  1 & \pmod {d} \\
x \equiv -1 &\pmod {(q-\chi)/d},
\end{cases}
\end{equation}
which always has a solution when $d$ satisfies the above conditions. In fact, since $\gcd(d, (q-\chi)/d)\mid 2$, there exist integers $u$ and $v$ such that $2=ud + v(q-\chi)/d$, and $n = 1 - u d = -1 + v(q-\chi)/d$ is a solution. Furthermore, the solution is unique modulo 
$$\lcm\left(d,(q-\chi)/d \right) = \begin{cases}
q-\chi & \text{ if } \gcd(d, (q-\chi)/d)=1\\
(q-\chi)/2 & \text{ if } \gcd(d, (q-\chi)/d)=2.
\end{cases}$$
 When $\nu_2(n-1) = \nu_2(q-\chi)-1$, there is exactly one solution that satisfies  Theorem~\ref{pcycle_inv_syst}(ii). Therefore, for each even divisor $d$ of $q-\chi$ such that $\nu_{2}(d) \in \{1, \nu_2(q-\chi)-1, \nu_2(q-\chi)\}$ and $\gcd(d,(q-\chi)/d)\mid 2$, there are one or two involutions $R_{n,a}$ with $d+\chi+1$ fixed points for every $a$ with $\chi(a)=\chi$.

\begin{proposition}\label{2-adm}
	A proper divisor $d$ of $q-\chi$ is $(q,\chi,2)$-admissible if and only if $d$ is even, $\nu_{2}(d) \in \{1, \nu_2(q-\chi)-1, \nu_2(q-\chi)\}$, and $\gcd(d,(q-\chi)/d)\mid 2$.  
\end{proposition}

The next result provides an explicit formula for $n$ such that $R_{n,a}$ is an involution with a prescribed number of fixed points. As already mentioned in the Introduction, if $\chi(a)=1$, then the cycle structures of $x^n$ and $R_{n,a}$ (as permutations over $\fq$ and $\pq$, respectively) are essentially the same with the only difference being that $R_{n,a}$ has one additional  fixed point. In particular,  $x^n$ is an involution if and only if $R_{n,a}$ is an involution. In~\cite{MR3651299} the authors give formulas for $n$ for which $x^n$ is an involution. The formulas presented in following proposition are the same as theirs when $\chi=1$.  

\begin{proposition}\label{inv_formulas}
	If $d$ is $(q,\chi,2)$-admissible, then $R_{n,a}$ is an involution with $d+\chi+1$ fixed points if and only if 
	$n \equiv k(q-\chi)/d-1 \pmod{q-\chi}$, where
	\[ k = \begin{cases}
	2\left( \dfrac{q-\chi}{d} \right)^{\varphi(d)-1} & \text{ if } \nu_2(d)=\nu_2(q-\chi)\\
	\left( \dfrac{q-\chi}{2d} \right)^{\varphi(d)-1} + \dfrac{d}{2} & \text{ if }  \nu_2(d)=\nu_2(q-\chi)-1\geq 1\\
	\left( \dfrac{q-\chi}{2d} \right)^{\varphi(d)-1} \text{ or } \left( \dfrac{q-\chi}{2d} \right)^{\varphi(d)-1} +\dfrac{d}{2} &\text{ if } \nu_2(d)=1, \nu_2(q-\chi)\geq 3\\
	\end{cases} \] 
	and $k$ is reduced modulo $d$. 
\end{proposition}

\begin{proof}
	By Theorem~\ref{pcycle_inv_syst}, we need to show that the given values of $n$ are all the solutions to~\eqref{syst} satisfying (ii) (when the condition applies). The case analysis relies on the values of $\nu_2(d)$ provided by Proposition~\ref{2-adm}.
	
	First, we suppose $\nu_2(d) =  \nu_2(q-\chi)$. Then $\gcd(d, (q-\chi)/d) =1$ and the solution to~\eqref{syst} is unique modulo $\lcm(d,(q-\chi)/d) = q-\chi$. For $k= 2\left( (q-\chi)/d \right)^{\varphi(d)-1}$, we compute  
	\[ \dfrac{k(q-\chi)}{d}-1 \equiv 2 \left(\dfrac{q-\chi}{d}\right)^{\varphi(d)} -1 \equiv 
	\begin{cases}
    1 &\pmod d \\
    -1 &\pmod{\dfrac{q-\chi}{d}}.
	\end{cases}\]
	The condition (ii) does not apply in this case, so $n \equiv k(q-\chi)/d-1\pmod{q-\chi}$ yields the unique involution with $d+\chi+1$ fixed points.
	
	If $\nu_2(d)< \nu_2(q-\chi)$, then the solution to~\eqref{syst} is unique modulo $\lcm(d,(q-\chi)/d) = (q-\chi)/2$, so there are two solutions modulo $q-\chi$.
	First, assume $\nu_2(d)= \nu_2(q-\chi)-1\geq 1$. Then $\gcd(d,(q-\chi)/d) = 2$ and $\gcd(d,(q-\chi)/2d) = 1$. One can check that for $k= \left( (q-\chi)/(2d) \right)^{\varphi(d)-1} + d/2$ the solutions to~\eqref{syst} are  $$n=\dfrac{k(q-\chi)}{d}-1 = 2\left(\dfrac{q-\chi}{2d} \right)^{\varphi(d)} + \dfrac{q-\chi}{2} -1$$ and $$n'=n - \dfrac{q-\chi}{2} = 2\left( \dfrac{q-\chi}{2d} \right)^{\varphi(d)} -1.$$ However, since  
	\[ n'-1 \equiv 2\left(\left( \dfrac{q-\chi}{2d}\right)^{\varphi(d)} -1 \right) \pmod{q-\chi}\]
	and $d \mid \left( (q-\chi)/(2d) \right)^{\varphi(d)} -1$, it follows that $n'-1$ is divisible by $2^{ \nu_2(q-\chi)}$. Therefore this solution does not satisfy (ii). 
	Now write $q-\chi = 2^{ \nu_2(q-\chi)}m$ with $m$ odd, and $n'-1 = 2^{ \nu_2(q-\chi)}\ell$. We compute
	\begin{align*}
	n-1 &= n'+ \frac{q-\chi}{2} -1\\
	&= 2^{ \nu_2(q-\chi)}\ell + \frac{2^{ \nu_2(q-\chi)}m}{2} \\ & = 2^{ \nu_2(q-\chi)-1}(2\ell+m),
	\end{align*}
	so $\nu_2(n-1)= \nu_2(q-\chi)-1$. 
	
	Finally, if $\nu_2(d)=1$ and $ \nu_2(q-\chi) \geq 3$, then $n$ and $n'$ are the two solutions to~\eqref{syst}.  \end{proof}

As a consequence we have the following.

\begin{proposition}\label{involutions_Md}
	Let $d$ be $(q,\chi,2)$-admissible and $M_d$ be the number of R\'edei  involutions with $d+\chi+1$ fixed points and a fixed parameter $a$ with $\chi(a)=\chi$. Then 
	\[ M_d = \begin{cases}
	1& \text{ if } \nu_2(d)=\nu_2(q-\chi)  \text{ or } \nu_2(d)=\nu_2(q-\chi)-1\geq 1 \\
	2 & \text{ if }  \nu_2(d)=1 \text{ and } \nu_2(q-\chi)\geq 3.
	\end{cases} \]
\end{proposition}

The following proposition gives the number of R\'edei involutions for a fixed $a$. 

\begin{proposition}\label{invcases}
	Let $q-\chi= 2^{\alpha_0}p_1^{\alpha_1}\cdots p_r^{\alpha_r} $ be the prime factorization of $q-\chi$, and $M$ be the number of R\'edei   involutions with a fixed parameter $a$ with $\chi(a)=\chi$. Then
	\[ M = \begin{cases}
	2^r -1 & \text{ if } \alpha_0=1 \\
	2^{r+1} -1 & \text{ if } \alpha_0=2 \\
	2^{r+2} -1 & \text{ if } \alpha_0\geq 3.
	\end{cases} \] 
\end{proposition}

\begin{proof} 
	The number of $(q,\chi,2)$-admissible divisors of $q-\chi$ depends on $\nu_2(q-\chi) = \alpha_0$. 
	
	If $\alpha_0=1$, there are $2^r$  integers $d$ that are $(q,\chi,2)$-admissible with $\nu_2(d) = \nu_2(q-\chi)$, so $M_d=1$. 
	
	If $\alpha_0=2$, then $\nu_2(d)$ is either $1=\nu_2(q-\chi) -1$ or $2=\nu_2(q-\chi)$. There are $2^{r+1}$  integers $d$ that are $(q,\chi,2)$-admissible, and for each one of them we have $M_d=1$.
	
	If $\alpha_0 \geq 3$, then $\nu_2(d)$ is 1, $\nu_2(q-\chi)-1$ or $\nu_2(q-\chi)$. For each $(q,\chi,2)$-admissible $d$ such that $\nu_2(d) \in \{1,\alpha_0-1\}$, we have $M_d=1$. Since there are $2^{r+1}$ such integers $d$, there are $2^{r+1}$ involutions. On the other hand, $M_d=2$ for each $(q,\chi,2)$-admissible $d$ with $\nu_2(d) = \alpha_0$. There are $2^r$ such integers $d$, adding up to $2^{r+1}$ involutions. We conclude that there are  $2\cdot 2^{r+1} = 2^{r+2}$ involutions.
	
	In all cases, we remove the identity from our counting since $d\neq q-\chi$.
  \end{proof}

We observe that by Proposition~\ref{invcases} involutions always exist except when $q=3$ and $\chi=1$. In addition, when $\alpha_0 = \nu_2(q-\chi) \in \{1,2\}$, the $M$ involutions have distinct cycle structures, that is, distinct numbers of fixed points. When $\alpha_0\geq 3$,  among the $2^{r+2}$ involutions there are $2^r$ pairs with the same cycle structure.

By Proposition~\ref{involutions_Md}, two involutions corresponding to an element with a certain quadratic character $\chi$ and the same cycle structure exist only when $\nu_2(q-\chi) \geq 3$ and $\nu_2(d)=1$. Corollary~\ref{coroinviso} shows that, if that is the case,  there is no involution with that cycle structure corresponding to an element of quadratic character $-\chi$. In order to prove this fact, we need the following observation.

\begin{lemma}\label{divex}
If $\chi(a) \neq \chi(b)$, then $4 \mid q - \chi(a)$ if and only if $2 \pdiv q - \chi(b)$.
\end{lemma}
\begin{proof}
Since $\chi(a) \neq \chi(b)$, the result follows from the fact that $(q - \chi(a))/2$ and $(q - \chi(b))/2$ are consecutive integers; so one is odd if and only if the other one is even.
  \end{proof}

\begin{corollary}\label{coroinviso}
Let $\chi(a)\neq \chi(b)$, $\nu_2(q-\chi(a)) \geq 3$, and $\nu_2(\gcd(n-1,q-\chi(a)))=1$. If $R_{n,a}$ and $R_{m,b}$ are involutions, then they have distinct cycle structures.
\end{corollary}

\begin{proof}
Let $d = \gcd(n-1,q-\chi(a))$ and $d'= \gcd(m-1,q-\chi(b))$. By Lemma~\ref{divex}, we have  $2 \pdiv q - \chi(b)$. So $\nu_2(d') =1$. If $R_{n,a}$ and $R_{m,b}$ have the same number of fixed points, then
\[d +\chi(a)+1 = d' +\chi(b)+1,\]
which implies that $(d-d')/2=(\chi(b)-\chi(a))/2=\pm 1$.  This is a contradiction, since $(d-d')/2$ is even.
  \end{proof}

\subsection{Involutions with few fixed points}

We now study involutions with few fixed points, which are preferable for applications like S-boxes in block ciphers~\cite{MR3480112,YTH}. The minimum number of fixed points attained by a R\'edei permutation is two. This is only possible when $\chi=-1$ and $d=2$, which is  $(q,\chi,2)$-admissible, except when $q=3$ and $\chi=1$. Proposition~\ref{inv_formulas} gives the following.

\begin{proposition}
	The R\'edei involutions $R_{n,a}$ with two  fixed points are given by \[n = \begin{cases}
	q & \text{ if } \nu_2(q+1) < 3 \\
	\dfrac{q-1}{2} \text{ or } q & \text{ if } \nu_2(q+1)\geq 3
	\end{cases} \]
	and $\chi(a)=-1$.
	Furthermore, the fixed points are $0$ and $\infty$.
\end{proposition}

For fixed $q$, there are one or two involutions over $\mathbb{P}^1(\mathbb{F}_q)$ with four fixed points.

\begin{proposition} 
	The R\'edei involutions $R_{n,a}$ with four  fixed points are given by  
	$$
	n=
	\begin{cases}
	\dfrac{q-3}{2} \text{ or } q-2 & \text{ if } q\equiv 1 \pmod 8  \text{ and }\chi(a)=1  \\
        q-2 & \text{ if }  q\equiv 3, 5, 7\pmod 8, \;q>3 \text{ and } \chi(a)=1 \\
        \dfrac{q-1}{2} & \text{ if }  q\equiv 3\pmod 8 \text{ and } \chi(a)=-1 \\
         \dfrac{3q-1}{4} &\text{ if }  q\equiv 7\pmod{32}  \text{ and } \chi(a)=-1 \\
         \dfrac{q-3}{4}& \text{ if } q\equiv 23\pmod{32}  \text{ and } \chi(a)=-1.
	\end{cases}
	$$
	Furthermore, when $\chi(a)=1$, the four fixed points are $0, \infty, -\sqrt{a}$ and $\sqrt{a}$. 
\end{proposition}

\begin{proof}
    The result follows directly from Proposition~\ref{inv_formulas} with the observation that involutions  with four fixed points only occur when $\chi=1$ and $d=2$, or $\chi=-1$ and $d=4$. The latter condition implies that $\nu_2(q+1)\in\{2,3\}$. 
  \end{proof}

When the fixed points of the involution are known, it is possible to reduce or completely remove them while keeping the involution property~\cite{MR3480112}. An explicit expression for the fixed points of a R\'edei function is given in~\cite{Chubb2019}.  One could apply such formulas  to remove the fixed points and obtain involutions with few fixed points. 

\subsection{Cycles of length $1$ and $4$}\label{cycles_length_4}

Let $j$ be a composite integer. A R\'edei permutation $R_{n,a}$ decomposes into $1$- and $j$-cycles if and only if $\gcd(n^j-1,q-\chi) = q-\chi$ and $\gcd(n^k-1,q-\chi) = \gcd(n-1,q-\chi)$ for all nontrivial divisors $k$ of $j$. 

\begin{lemma}\label{jadm}
	Any $(q,\chi,j)$-admissible integer is $(q,\chi,k)$-admissible when $k \mid j$. 
\end{lemma}

\begin{proof}
	If $R_{n,a}$ is a permutation with $1$- and $j$-cycles and $j=k \ell$, then $R_{n^{\ell},a}$ is a permutation with $1$- and $k$-cycles. Furthermore, both of them have the same number of fixed points.
  \end{proof}

The above observations imply that  $R_{n,a}$ has  $1$- and $4$-cycles if and only if $\gcd(n^4-1,q-\chi) = q-\chi$ and  $\gcd(n^2-1,q-\chi) = \gcd(n-1,q-\chi)$. Also, any $(q,\chi,4)$-admissible integer is  $(q,\chi,2)$-admissible. 

\begin{proposition}\label{4-admissible}
	A proper divisor $d$ of $q-\chi$ is $(q,\chi,4)$-admissible if and only if $\nu_2(d) = \nu_2(q-\chi)$, $\nu_p(d) = \nu_p(q-\chi)$ for all primes $p$ of the form $4k+3$, and $\gcd(d, (q-\chi)/d) = 1$.
\end{proposition}

\begin{proof}
	Suppose that $R_{n,a}$ is a permutation with $d+\chi+1$ fixed points and $4$-cycles. Then $\gcd(n^4-1,q-\chi) = q-\chi$ and $\gcd(n^2-1,q-\chi) = \gcd(n-1,q-\chi) = d$. The last equality, along with the fact that $n^4-1 = (n-1)(n+1)(n^2+1)$ and $\gcd(n+1,q-\chi) \geq 2$, imply that $\nu_{2}(d) = \nu_2(q-\chi)$. 
	
	Since $R_{n,a} \circ R_{n,a} = R_{n^2,a}$, it follows that $R_{n^2,a}$ is an involution, and both permutations have the same number of fixed points. Consequently, $d$ is $(q,\chi,2)$-admissible, and hence must satisfy the conditions in Proposition~\ref{2-adm}.
	
	Furthermore, $R_{n^2,a}$ being an involution implies that $n^2\equiv -1 \pmod{(q-\chi)/d}$, so $-1$ is a quadratic residue modulo $(q-\chi)/d$. Since $(q-\chi)/d$ is odd, it does not have any prime factor of the form $4k+3$. Thus if any prime of this form divides $q-\chi$, it must divide $d$. 
	
	Conversely, we assume that $d$ is such that $\nu_2(d) = \nu_2(q-\chi)$, $\nu_p(d) = \nu_p(q-\chi)$ for all primes $p$ of the form $4k+3$, and $\gcd(d, (q-\chi)/d) = 1$. Then $d$ is $(q,\chi,2)$-admissible, and there exists $m$ such that $R_{m,a}$ is an involution with $d+\chi+1$ fixed points. It follows that $\gcd(m-1,q-\chi)=d$ and $m\equiv -1\pmod{(q-\chi)/d}$. The assumption on the factorization of $(q-\chi)/d$ guarantees that $m$ is a quadratic residue modulo $(q-\chi)/d$. Let $\ell$ be an integer such that $\ell^2 \equiv m \equiv -1 \pmod{(q-\chi)/d}$. Since $\gcd(d,(q-\chi)/d) =1$, there exists an integer $n$ such that $n\equiv 1 \pmod d$ and $n\equiv \ell \pmod{(q-\chi)/d}$. It follows that $\gcd(n^4-1,q-\chi) = q-\chi$ and $\gcd(n-1,q-\chi) = \gcd(n^2-1,q-\chi) =d$. Therefore $R_{n,a}$ is a permutation with $1$- and $4$-cycles. 
  \end{proof}

The proof of the proposition gives a procedure to find all R\'{e}dei permutations with a prescribed number of fixed points and $4$-cycles.

\begin{enumerate}[1.]
	\item  Compute $m$ according to the formula given in Proposition~\ref{inv_formulas}. 
	\item Find the square roots of $m$ modulo $(q-\chi)/d$.
	\item For each square root $\ell$ of $m$, find all integers $n$ modulo $q-\chi$ that satisfy 
	\[\begin{cases}n \equiv  1 & \pmod {d} \\
	n \equiv \ell &\pmod {(q-\chi)/d}.\end{cases}\]
\end{enumerate}

The above procedure and the ones for finding R\'edei permutations with $1$- and $p$-cycles for $p \in \{3,5\}$ in  Section~\ref{cycleslength3}  only require solving linear congruences and computing square roots modulo an integer. The computation of square roots of an integer $b$ modulo $m = p_1^{\ell_1}\cdots p_s^{\ell_s}$ can be carried out as follows.
\begin{enumerate}[1.]
    \item For $i \in \{1, \dots, s\}$, do the following.
    \begin{enumerate}
        \item Compute the square roots of $b$ modulo $p_i$. If $p_i \equiv 3 \pmod 4$, they are given by $\pm b^{(p_i+3)/4} \pmod{p_i}$. If $p_i \equiv 1 \pmod 4$, then Shanks' Algorithm~\cite{MR0371855} can be used.
        \item Use Hensel's Lemma~\cite[Theorem 2.23]{MR1083765} to compute the square roots modulo $p_i^{\ell_i}$.
    \end{enumerate}
    \item Combine the square roots modulo distinct prime powers using the Chinese Remainder Theorem.
\end{enumerate}

The next results concern an  existence condition and the number of R\'edei permutations with $1$- and $4$-cycles. We omit the proofs as they are analogous to their corresponding ones in Section~\ref{cycleslengthp}.

\begin{proposition}\label{exist4}
    There exists a R\'edei permutation over $\pq$ with $1$- and $4$-cycles if and only if $q-1$ or $q+1$ has a prime factor of the form $4k+1$. 
\end{proposition}

\begin{proposition}\label{4cycle_number}
	Let $d$ be $(q,\chi,4)$-admissible and let $M_d$ be the number of R\'edei   permutations $R_{n,a}$ with $d+\chi+1$ fixed points and $4$-cycles for a fixed parameter $a$ with $\chi(a)=\chi$. Then $ M_d= 2^{u}$, where $u= |\{p\ prime\colon p\equiv 1 \pmod 4, p\mid q-\chi \ and\ p\nmid d\}|$.
\end{proposition}

\begin{proposition}
	Let $M$ be the number of R\'edei   permutations $R_{n,a}$ with $1$- and $4$-cycles for a fixed parameter $a$ with $\chi(a)=\chi$. Then
	$ M= 3^{r}-1$, where $r= |\{p\ prime\colon p\equiv 1 \pmod 4, p\mid q-\chi\}|$.
\end{proposition}

\section{Constructions of R\'edei permutations with $1$- and $p$-cycles for $p\in\{3,5\}$}\label{cycleslength3}

\subsection{Cycles of length $1$ and $3$}\label{1and3}

Applying the results in Section~\ref{cycleslengthp} to $p=3$ gives the following characterization of R\'edei   permutations with $1$- and $3$-cycles.

\begin{proposition}
	Let $q-\chi = p_1^{\alpha_1} \cdots p_r^{\alpha_r}$ and $d = p_1^{\beta_1} \cdots p_r^{\beta_r} < q - \chi$ with 
	\[\beta_i = \begin{cases}
	\alpha_i-1 \text{ or } \alpha_i & \text{ if }p_i = 3 \\
	0 \text{ or } \alpha_i & \text{ if } p_i \equiv 1\pmod 3 \\ 
	\alpha_i & \text{ if } p_i \equiv 2\pmod 3 .
	\end{cases}\]
	The R\'edei   permutation $R_{n,a}$ has $d+\chi+1$ fixed points and $3$-cycles if and only if $n \equiv 1 \pmod d$, $n^2+n+1\equiv 0 \pmod{(q-\chi)/d}$, and if $\nu_3(d) = \nu_3(q-\chi)-1$ then $\nu_3(n-1)= \nu_3(d)$. Moreover, these are all the R\'edei permutations with $1$- and $3$-cycles.
\end{proposition}

In order to find R\'edei permutations with $1$-  and $3$-cycles, one needs to solve  \begin{equation}\label{quadratic_congruence} x^2+x+1\equiv 0 \pmod{(q-\chi)/d}. \end{equation} 
Since $d$ is even, $(q-\chi)/d$ is odd. By letting $y=2x+1$, the congruence \eqref{quadratic_congruence} can be expressed as 
\begin{equation}\label{quadratic_cong2} y^2 \equiv -3 \pmod{(q-\chi)/d},\end{equation} 
so that the solutions to~\eqref{quadratic_congruence} are in bijection with the solutions to~\eqref{quadratic_cong2}. Thus solving the quadratic congruence is equivalent to finding the square roots of $-3$ modulo $(q-\chi)/d$. 

We have the following procedure to construct all R\'edei permutations with $d+\chi+1$ fixed points and $(q-d-1)/3$ cycles of length $3$.

\begin{enumerate}[1.]
	\item Find all integers $y$ modulo $(q-\chi)/d$ such that $y^2 \equiv -3 \pmod{(q-\chi)/d}$.
	\item For each $y$, find all integers $n$ modulo $q-\chi$ that satisfy \[\begin{cases}
	n \equiv 1 & \pmod{d} \\
	2n \equiv y-1 & \pmod{(q-\chi)/d}.
	\end{cases}\]
	\item If $\nu_3(d) = \nu_3(q-\chi)-1$, discard all $n$ such that $ \nu_3(n-1) \geq \nu_3(q-\chi)$.
\end{enumerate}

Proposition~\ref{p-admissible} implies that, for fixed $q$ and $\chi$, among the permutations with $1$- and $3$-cycles the minimum number of fixed points occurs when $d$ is not divisible by any prime of the form $3k+1$. However, $d$ is necessarily divisible by all primes of a different shape that appear in the factorization of $q-\chi$. So the minimum number of fixed points is the product of all prime powers in the factorization of $q-\chi$ for primes of the form $3k+2$. In the particular case when $(q-\chi)/2$ has no prime factor of the form $3k+2$, there are permutations with as few as two fixed points if $\chi =-1$, and four fixed points if $\chi =1$. 

When $(q-\chi)/2$ is a prime number of the form $4k+3$, there is a formula attributed to Lagrange to compute the square roots of $-3$ modulo $(q-\chi)/2$. This results in an  explicit expression for $n$.

\begin{proposition}
	Suppose $q-\chi = 2p$, where $p \equiv 7 \pmod{12}$. Then $R_{n,a}$ has $1$- and $3$-cycles if and only if 
	\begin{equation}\label{3cycle_2fp} n \equiv \dfrac{p+1}{2}\left( \pm (-3)^{(p+1)/4}\right) +p \pmod{q-\chi}. \end{equation}
	Furthermore, $R_{n,a}$ has two fixed points, $0$ and $\infty$, if $\chi =-1$,  and four fixed points, $0$, $\infty$, $-\sqrt a$ and $\sqrt a$, if $\chi=1$.
\end{proposition}

\begin{proof}
	In this case $2$ is the only $(q,\chi,3)$-admissible integer, and the system \eqref{pcycle_syst} becomes 
	\begin{equation*} 
	\begin{cases}
	x \equiv  1 & \pmod {2} \\
	x^2 +x +1\equiv 0 &\pmod{p}.
	\end{cases}
	\end{equation*}
	The requirement $p \equiv 7 \pmod{ 12}$ guarantees that $p \equiv 1 \pmod 3$ and $p \equiv 3 \pmod 4$, thus $-3$ is a quadratic residue modulo $p$ and a formula is known for its square roots. It is easy to check that the solutions are the ones given in \eqref{3cycle_2fp}. By Proposition~\ref{pcycle_number_total}, these are all the permutations with $1$- and $3$-cycles.
  \end{proof}

As a consequence of Theorem~\ref{pcycle_inv_syst}, we also have the following.

\begin{corollary}
	If $R_{n,a}$ has  $d+\chi+1$ fixed points and $3$-cycles, then so  has $R_{n^2,a}$. 
\end{corollary}

\begin{proof}
	Assume that $R_{n,a}$ has $d+\chi+1$ fixed points and $3$-cycles. Then Conditions (i) and (ii) in Theorem~\ref{pcycle_inv_syst} hold for $n$. We have to show that they also hold for $n^2$. Clearly $n^2\equiv 1 \pmod d$. Furthermore, $$n^4+n^2+1 \equiv n^2(n^2+1)+1 \equiv n^2(-n) +1 \equiv -n^3+1 \equiv 0 \pmod{(q-\chi)/d}.$$ Finally, if $\nu_3(d) = \nu_3(q-\chi)-1$, then $\nu_3(n-1) = \nu_3(d)$. Since $n^2-1 = (n-1)(n+1)$ and the only common factor of $n-1$ and $n+1$ is 2, we conclude that  $\nu_3(n^2-1) = \nu_3(d)$. 
  \end{proof}

\subsection{Cycles of length $1$ and $5$}\label{1and5}

Proceeding analogously to the case $p=3$, we now apply the characterization given in Theorem~\ref{pcycle_inv_syst} to $p=5$. In this case the second congruence in~\eqref{pcycle_syst} becomes 
\begin{equation} \label{cyclotomic5}
		n^{4}+n^{3}+n^2+n+1\equiv 0 \pmod {(q-\chi)/d}.
\end{equation}
Dividing by $n^2$ and making the substitution $y=n+n^{-1}$, Equation~\eqref{cyclotomic5} becomes 
\begin{equation}\label{cyclotomic5_change1}
    y^2+y-1 \equiv 0 \pmod {(q-\chi)/d}.
\end{equation}
Similarly to what we did with $p=3$, we can now solve this quadratic congruence by making the change of variables $z=2y+1$ and obtain

\begin{equation}\label{cyclotomic5_change2}
    z^2 \equiv 5 \pmod {(q-\chi)/d}.
\end{equation}

Our strategy to find $n$ then consists in first finding a solution $z$ to~\eqref{cyclotomic5_change2}, solving the linear congruence $2y\equiv z-1 \pmod{(q-\chi)/d}$, and  obtain an element $y$ that satisfies~\eqref{cyclotomic5_change1}. Now $y$ and $n$ are related by $y=n+n^{-1}$ or $n^2-yn+1=0$. We have to solve  $n^2-yn+1 \equiv 0 \pmod{(q-\chi)/d}$ for $n$, and again this can be done in two steps. The change of variables $m=2n-y$ transforms the last quadratic congruence into $m^2 \equiv y^2-4 \pmod{(q-\chi)/d}$. Once we find a solution $m$ to the last congruence, we solve $2n \equiv m+y \pmod{(q-\chi)/d}$ to finally obtain $n$.
Therefore we have the following procedure to find all R\'edei permutations with $d+\chi+1$ fixed points and $(q-d-1)/5$ cycles of length $5$.

\begin{enumerate}[1.]
    \item Find all integers $z$ modulo $(q-\chi)/d$ such that $z^2 \equiv 5 \pmod {(q-\chi)/d}$.
    \item For each $z$, find $y$ modulo $(q-\chi)/d$ such that $2y \equiv z-1 \pmod {(q-\chi)/d}$.
    \item For each $y$, find all integers $m$ modulo $(q-\chi)/d$ such that $m^2 \equiv y^2-4 \pmod {(q-\chi)/d}$.
    \item For each pair $(y,m)$, find all integers $n$ modulo $q-\chi$ such that \[\begin{cases}
	n \equiv 1 & \pmod{d} \\
	2n \equiv m+y & \pmod{(q-\chi)/d}. 
	\end{cases}\]
	\item If $\nu_5(d) = \nu_5(q-\chi)-1$, discard all $n$ such that $ \nu_5(n-1) \geq \nu_5(q-\chi)$.
\end{enumerate}

\begin{remark}\label{remark1}
For  primes $p > 5$, the trick of dividing the congruence by $n^{(p-1)/2}$ and making the substitution $y=n+n^{-1}$ reduces the degree of the polynomial by  a half. For instance, in the case $p=7$ the second congruence in~\eqref{pcycle_syst} is
$ n^6+n^5+n^4+n^3+n^2+n+1 \equiv 0 \pmod{(q-\chi)/d}$,
and employing the same trick produces
$y^3+y^2-2y-1 \equiv 0 \pmod{(q-\chi)/d}.$
\end{remark}

\section{Examples}\label{examples}

In this section we illustrate our results and produce all R\'edei permutations $R_{n,a}$ with $\chi(a)=\chi$ over ${\mathbb{P}^1}(\mathbb F_{q})$ that decompose into $1$- and $j$-cycles when $j$ is $4$ or prime for $q=125$ and $q=841$. We also describe their cycle decompositions. When $q=125$, we display our findings for $\chi=1$ and $-1$ in Tables 1 and 2, respectively. When $q=841$, we display our findings for $\chi=1$ and $-1$ in Tables 3 and 4, respectively. Each table is constructed as follows.

  For each $j$ we first check whether or not $R_{n,a}$ with $\chi(a)=\chi$, $1$- and $j$-cycles exists, by using Theorem~\ref{p-cycle-existence} (for an odd prime $j$) or  Proposition~\ref{exist4} (for $j=4$). The conditions appear in Columns 2 and 3. In the case $j=4$ only the condition in Column 2 applies. By Proposition~\ref{invcases} involutions always exist except when $q=3$ and $\chi=1$.  
 
 The values in Column 4 are all the $(q,\chi,j)$-admissible integers $d$, obtained by applying Proposition~\ref{p-admissible}, \ref{2-adm} or \ref{4-admissible}, depending on $j$. Then for each $d$, we display $M_d$ in Column 5, where $M_d$ is the number of permutations $R_{n,a}$ with $d+\chi+1$ fixed points, $j$-cycles, and a fixed parameter $a$ with $\chi(a)=\chi$. To obtain $M_d$ we use Proposition~\ref{pcycle_number}, \ref{involutions_Md} or \ref{4cycle_number}. In Column 6 we present all the $M_d$ values of $n$ for which $R_{n,a}$ is one of these permutations. These values are obtained by applying Proposition~\ref{inv_formulas} ($j=2$), the procedures presented in Subsections~\ref{cycles_length_4} ($j=4$), \ref{1and3} ($j=3$), \ref{1and5} ($j=5$), or Remark~\ref{remark1} ($j=7$).
 
For the integers $n$ in Column 6, Columns 7 and 8 display the number of fixed points and the number of $j$-cycles of the permutations $R_{n,a}$ with $\chi(a)=\chi$, respectively. The number of fixed points is $d+\chi+1$, and the number of $j$-cycles is $(q-d-\chi)/j$.  
 
\begin{table}[b]
\resizebox{\textwidth}{!}{%
 \begin{tabular}{|c | c |  c |  c | c| c | c | c|} 
 \hline
$j$ & \begin{tabular}{c}prime $jk+1$, \\$jk+1\mid q-1$\textup{?}\end{tabular} & $j^2\mid q-1$\textup{?} & $d$ & $M_d$ & $n$&\begin{tabular}{c} \# fixed\\points \end{tabular} & \# $j$-cycles \\
 \hline\hline
  & \multicolumn{2}{c|}{{\multirow{3}{*}{}}}   & $2$  & $1$ & $123$& $4$ & $61$\\ 
  \cline{4-8} 
 $2$ & \multicolumn{2}{c|}{\textup{N/A}}   & $4$  & $1$ & $61$& $6$ & $60$\\ 
 \cline{4-8} 
 & \multicolumn{2}{c|}{}  & $62$  & $1$ & $63$& $64$ & $31$\\ 
 \hline
 $3$ & \textup{yes}, $31$& \textup{no} & $4$  & $2$ & $5, 25 $& $6$ & $40$\\ 
 \hline
 $4$ & \textup{no}& \textup{N/A} & \multicolumn{5}{c|}{} \\ 
 \hline
 $5$ & \textup{yes}, $31$& \textup{no}&$4$ &$4$ &\begin{tabular}{c} $33, 97$ \\$101, 109$\end{tabular} & $6$ & $24$ \\
 \hline
\begin{tabular}{c} prime \\$\geq 7$\end{tabular} & \textup{no}& \textup{no} &  \multicolumn{5}{c|}{}\\
\hline
 \end{tabular}}
 \caption{R\'edei permutations with $1$- and $j$-cycles over $\mathbb{P}^1(\mathbb{F}_{125})$ with $\chi(a)=1$.}\label{tableq125chi1}
\end{table}

\begin{example}
We find all R\'edei permutations $R_{n,a}$ with $1$- and $j$-cycles  over ${\mathbb{P}^1}(\mathbb F_{125})$, where $j=4$ or a prime number. When $\chi(a)=1$, we have $q-\chi(a) = 2^2\cdot 31$. In this case, there are three involutions, two R\'edei permutations with $1$- and $3$-cycles, and four R\'edei permutations with $1$- and $5$-cycles. Table \ref{tableq125chi1} gives more details. For instance, the first three rows in Table \ref{tableq125chi1} can be transcribed in the following way. For a fixed field element $a\in\mathbb F_{125}$ with $\chi(a)=1$, there are precisely three R\'edei involutions $R_{n,a}$  over ${\mathbb{P}^1}(\mathbb F_{125})$. They are
\begin{itemize}
    \item $R_{123,a}$ with 4 fixed points and $61$ cycles of length $2$,
     \item  $R_{61,a}$ with 6 fixed points and $60$ cycles of length $2$,
      and 
      \item $R_{63,a}$ with 64 fixed points and $31$ cycles of length $2$.
\end{itemize}
\noindent When $\chi(a)=-1$, we have $q-\chi(a) = 2\cdot 3^2 \cdot 7$. In this case, there are three involutions and eight R\'edei permutations with $1$- and $3$-cycles; see Table \ref{tableq125chi-1}.

\begin{table}[t]
\resizebox{\textwidth}{!}{%
 \begin{tabular}{|c | c |  c |  c| c | c | c | c|} 
 \hline
$j$ & \begin{tabular}{c}prime $jk+1$, \\$jk+1\mid q+ 1$\textup{?}\end{tabular} & $j^2\mid q+1$\textup{?} & $d$ & $M_d$ & $n$ & \begin{tabular}{c} \# fixed\\points \end{tabular} & \# $j$-cycles \\ \hline\hline
 &\multicolumn{2}{c|}{\multirow{3}{*}{\textup{N/A}}} &$2$ &$1$&$125$&$2$& $62$\\ 
 \cline{4-8} 
$2$ &\multicolumn{2}{c|}{}  &$14$ & $1$ & $71$ & $14$  & $56$ \\ 
\cline{4-8} 
& \multicolumn{2}{c|}{} &$18$ & $1$ & $55$ & $18$ & $54$\\
 \hline
  \multirow{3}{*}{$3$} &&  & $6$  & $4$  & \begin{tabular}{c} $25, 67$ \\$79, 121$\end{tabular} & $6$ & $40$\\ 
  \cline{4-8} 
   &  \textup{yes}, $7$ &\textup{yes}&$18$ &$2$ & $37, 109$& $18$ & $36$ \\
   \cline{4-8} 
  & & &$42$ &$2$ & $43, 85$& $42$ & $28$ \\
 \hline
  $4$ & \textup{no}& \textup{N/A} & \multicolumn{5}{c|}{} \\ 
 \hline
\begin{tabular}{c} prime \\$\geq 5$\end{tabular} & \textup{no} & \textup{no} & \multicolumn{5}{c|}{} \\
\hline
 \end{tabular}}
  \caption{R\'edei permutations with $1$- and $j$-cycles over $\mathbb{P}^1(\mathbb{F}_{125})$ with $\chi(a)=-1$.}\label{tableq125chi-1}
\end{table}
\end{example}

\begin{table}[b]
\resizebox{\textwidth}{!}{%
\begin{tabular}{| c | c | c | c | c | c | c | c |} 
\hline
$j$ & \begin{tabular}{c}prime $jk+1$, \\$jk+1\mid q-1$\textup{?}\end{tabular} & $j^2\mid q-1$\textup{?} & $d$ & $M_d$ & $n$ & \begin{tabular}{c} \# fixed\\points \end{tabular} & \# $j$-cycles \\ \hline\hline
 &\multicolumn{2}{c|}{\multirow{24}{*}{\textup{N/A}}} &$2$ &$ 2$&$ 419, 839$ & $4 $& $ 419$\\ 
 \cline{4-8} 
 &\multicolumn{2}{c|}{}  &$ 4$ & $1$ & $ 629$ & $ 6$  & $ 418 $ \\ 
 \cline{4-8} 
& \multicolumn{2}{c|}{} &$6 $ & $ 2$ & $ 139, 559$ & $ 8$ & $ 417$\\
\cline{4-8} 
& \multicolumn{2}{c|}{} &$8 $ & $ 1$ & $ 209$ & $ 10$ & $ 416$\\
\cline{4-8} 
& \multicolumn{2}{c|}{} &$10 $ & $ 2$ & $ 251, 671$ & $ 12$ & $ 415$\\
\cline{4-8} 
& \multicolumn{2}{c|}{} &$12 $ & $ 1$ & $ 349$ & $ 14$ & $ 414$\\
\cline{4-8} 
& \multicolumn{2}{c|}{} &$14 $ & $ 2$ & $ 239,659$ & $ 16$ & $ 413$\\
\cline{4-8} 
& \multicolumn{2}{c|}{} &$20 $ & $ 1$ & $ 461$ & $ 22$ & $ 410$\\
\cline{4-8} 
& \multicolumn{2}{c|}{} &$24 $ & $ 1$ & $ 769$ & $ 26$ & $ 408$\\
\cline{4-8} 
& \multicolumn{2}{c|}{} &$28 $ & $ 1$ & $ 29$ & $ 30$ & $ 406$\\
\cline{4-8} 
& \multicolumn{2}{c|}{} &$30 $ & $ 2$ & $ 391, 811$ & $ 32$ & $ 405$\\
\cline{4-8} 
$2$ & \multicolumn{2}{c|}{} &$40 $ & $ 1$ & $ 41$ & $ 42$ & $ 400$\\
\cline{4-8} 
& \multicolumn{2}{c|}{} &$42 $ & $ 2$ & $ 379, 799$ & $ 44$ & $ 399$\\
\cline{4-8} 
& \multicolumn{2}{c|}{} &$56 $ & $ 1$ & $ 449$ & $ 58$ & $ 392$\\
\cline{4-8} 
& \multicolumn{2}{c|}{} &$60 $ & $ 1$ & $ 181$ & $ 62$ & $ 390$\\
\cline{4-8} 
& \multicolumn{2}{c|}{} &$70 $ & $ 2$ & $ 71, 491$ & $ 72$ & $ 385$\\
\cline{4-8} 
& \multicolumn{2}{c|}{} &$84 $ & $ 1$ & $ 589$ & $ 86$ & $ 378$\\
\cline{4-8} 
& \multicolumn{2}{c|}{} &$120 $ & $ 1$ & $ 601$ & $ 122$ & $360 $\\
\cline{4-8} 
& \multicolumn{2}{c|}{} &$140 $ & $ 1$ & $ 701$ & $ 142$ & $ 350$\\
\cline{4-8} 
& \multicolumn{2}{c|}{} &$168 $ & $ 1$ & $ 169$ & $ 170$ & $336 $\\
\cline{4-8} 
& \multicolumn{2}{c|}{} &$210 $ & $ 2$ & $ 211, 631$ & $ 212$ & $ 315$\\
\cline{4-8} 
& \multicolumn{2}{c|}{} &$280 $ & $ 1$ & $ 281$ & $ 282$ & $ 280$\\
\cline{4-8} 
& \multicolumn{2}{c|}{} &$420 $ & $ 1$ & $ 421$ & $ 422$ & $ 210$\\
 \hline
 $3$ &\textup{yes}, $7$ & \textup{no} & $120 $  & $2 $  & $121, 361$ & $122$ & $ 240$\\ 
 \hline
  $4$ &\textup{yes}, $ 5$ & \textup{N/A} & $168 $  & $2 $  & $337, 673$ & $ 170$ & $ 168$\\ 
 \hline
\begin{tabular}{c} prime \\$\geq 5$\end{tabular} & \textup{no} & \textup{no} & \multicolumn{5}{c|}{} \\
\hline
\end{tabular}}
 \caption{R\'edei permutations with $1$- and $j$-cycles over $\mathbb{P}^1(\mathbb{F}_{841})$ with $\chi(a)=1$.}\label{tableq841chi1}
\end{table}

\begin{example} We find all R\'edei permutations $R_{n,a}$ with $1$- and $j$-cycles  over ${\mathbb{P}^1}(\mathbb F_{841})$,  where $j=4$ or a prime number.
%$q= 29^2= 841$. 
When $\chi(a)=1$, we have  $q-\chi(a) = 2^3 \cdot 3\cdot 5\cdot 7$. In this case, there are 31 involutions, two R\'edei permutations with $1$- and $3$-cycles, and two R\'edei permutations with $1$- and $4$-cycles; see Table~\ref{tableq841chi1}.  
 \noindent When $\chi(a)=-1$, we have $q-\chi(a) = 2\cdot 421$. In this case, we have one involution, two R\'edei permutations with $1$- and $3$-cycles, two R\'edei permutations with $1$- and $4$-cycles, four R\'edei permutations with $1$- and $5$-cycles, and six R\'edei permutations with $1$- and $7$-cycles; see Table~\ref{tableq841chi-1}.

\begin{table}[h]
\resizebox{\textwidth}{!}{%
\begin{tabular}{| c | c | c | c | c | c | c | c |} 
\hline
$j$ & \begin{tabular}{c}prime $jk+1$, \\$jk+1\mid q+1$\textup{?}\end{tabular}  & $j^2\mid q+1$\textup{?} & $d$ & $M_d$ & $n$&\begin{tabular}{c} \# fixed\\points \end{tabular} & \# $j$-cycles \\
\hline\hline
$2$ & \multicolumn{2}{c|}{\textup{N/A}}& $2$ & $1$ & $ 841$ & $2$ & $420$\\ 
\hline
$3$ & \textup{yes}, $421$& \textup{no} & $2$  & $ 2$ & $441, 821 $& $2 $ & $280$\\  
 \hline
  $4$ &\textup{yes}, $421 $  & \textup{N/A} & $2 $  & $2 $  & $ 29,  813$ & $ 2$ & $ 210$ \\ 
  \hline
  $5$ & \textup{yes}, $421$& \textup{no}&$ 2$ &$4$ &\begin{tabular}{c}
       $ 279, 377,$ \\
       $673, 775 $ 
  \end{tabular} & $2$ & $ 168$ \\
 \hline
     $7$&\textup{yes}, $421$&\textup{no}&$ 2$ &$6 $ &\begin{tabular}{c}$ 33, 75,$ \\$ 247, 385,$ \\ $ 573, 791$\end{tabular} & $2$ & $ 120$ \\
 \hline
\begin{tabular}{c} prime \\$\geq 11$\end{tabular} & \textup{no}& \textup{no} &  \multicolumn{5}{c|}{}\\
\hline
 \end{tabular}}
  \caption{R\'edei permutations with $1$- and $j$-cycles over $\mathbb{P}^1(\mathbb{F}_{841})$ with $\chi(a)=-1$.}\label{tableq841chi-1}
\end{table}
\end{example}

%\bibliographystyle{acm}   %agsm ou plain ou acm
%\bibliography{bib_permutations}

\begin{thebibliography}{10}
	
	\bibitem{AHMAD1969370}
	{\sc Ahmad, S.}
	\newblock Cycle structure of automorphisms of finite cyclic groups.
	\newblock {\em J. Combinatorial Theory 6}, 4 (1969), 370 -- 374.
	
	\bibitem{MR2766784}
	{\sc Barbero, S., Cerruti, U., and Murru, N.}
	\newblock Solving the {P}ell equation via {R}\'{e}dei rational functions.
	\newblock {\em Fibonacci Quart. 48}, 4 (2010), 348--357.
	
	\bibitem{MR3475548}
	{\sc Bellini, E., and Murru, N.}
	\newblock An efficient and secure {RSA}-like cryptosystem exploiting
	{R}\'{e}dei rational functions over conics.
	\newblock {\em Finite Fields Appl. 39\/} (2016), 179--194.
	
	\bibitem{MR3651299}
	{\sc Castro, F.~N., Corrada-Bravo, C., Pacheco-Tallaj, N., and Rubio, I.}
	\newblock Explicit formulas for monomial involutions over finite fields.
	\newblock {\em Adv. Math. Commun. 11}, 2 (2017), 301--306.
	
	\bibitem{MR2435050}
	{\sc \c{C}e\c{s}melio\u{g}lu, A., Meidl, W., and Topuzo\u{g}lu, A.}
	\newblock On the cycle structure of permutation polynomials.
	\newblock {\em Finite Fields Appl. 14}, 3 (2008), 593--614.
	
	\bibitem{MR3587256}
	{\sc Charpin, P., Mesnager, S., and Sarkar, S.}
	\newblock Dickson polynomials that are involutions.
	\newblock In {\em Contemporary developments in finite fields and applications}.
	World Sci. Publ., Hackensack, NJ, 2016, pp.~22--47.
	
	\bibitem{MR3480112}
	{\sc Charpin, P., Mesnager, S., and Sarkar, S.}
	\newblock Involutions over the {G}alois field {$\mathbb{F}_{2^n}$}.
	\newblock {\em IEEE Trans. Inform. Theory 62}, 4 (2016), 2266--2276.
	
	\bibitem{Chubb2019}
	{\sc Chubb, K., Panario, D., and Wang, Q.}
	\newblock Fixed points of rational functions satisfying the {C}arlitz property.
	\newblock {\em Appl. Algebra Engrg. Comm. Comput. 30}, 5 (2019), 417--439.
	
	\bibitem{MR3954416}
	{\sc Fu, S., Feng, X., Lin, D., and Wang, Q.}
	\newblock A recursive construction of permutation polynomials over
	{$\mathbb{F}_{q^2}$} with odd characteristic related to {R}\'{e}dei
	functions.
	\newblock {\em Des. Codes Cryptogr. 87}, 7 (2019), 1481--1498.
	
	\bibitem{MR2401984}
	{\sc Gutierrez, J., and Winterhof, A.}
	\newblock Exponential sums of nonlinear congruential pseudorandom number
	generators with {R}\'{e}dei functions.
	\newblock {\em Finite Fields Appl. 14}, 2 (2008), 410--416.
	
	\bibitem{key:article}
	{\sc Kameswari, P.~A., and Kumari, R.~C.}
	\newblock Cryptosystem with redei rational functions via pellconics.
	\newblock {\em International Journal of Computer Applications 54}, 15
	(September 2012), 1--6.
	
	\bibitem{MR1159877}
	{\sc Lidl, R., and Mullen, G.~L.}
	\newblock Cycle structure of {D}ickson permutation polynomials.
	\newblock {\em Math. J. Okayama Univ. 33\/} (1991), 1--11.
	
	\bibitem{MR1542258}
	{\sc Lidl, R., and Mullen, G.~L.}
	\newblock Unsolved {P}roblems: {W}hen {D}oes a {P}olynomial over a {F}inite
	{F}ield {P}ermute the {E}lements of the {F}ield?, {II}.
	\newblock {\em Amer. Math. Monthly 100}, 1 (1993), 71--74.
	
	\bibitem{MR3087321}
	{\sc Mullen, G.~L.}, Ed.
	\newblock {\em Handbook of finite fields}.
	\newblock Discrete Mathematics and its Applications (Boca Raton). CRC Press,
	Boca Raton, FL, 2013.
	
	\bibitem{MR3784184}
	{\sc Murru, N., and Saettone, F.~M.}
	\newblock A novel {RSA}-like cryptosystem based on a generalization of the
	{R}\'{e}dei rational functions.
	\newblock In {\em Number-theoretic methods in cryptology}, vol.~10737 of {\em
		Lecture Notes in Comput. Sci.} Springer, Cham, 2018, pp.~91--103.
	
	\bibitem{MR4071834}
	{\sc Niu, T., Li, K., Qu, L., and Wang, Q.}
	\newblock New constructions of involutions over finite fields.
	\newblock {\em Cryptogr. Commun. 12}, 2 (2020), 165--185.
	
	\bibitem{MR1083765}
	{\sc Niven, I., Zuckerman, H.~S., and Montgomery, H.~L.}
	\newblock {\em An introduction to the theory of numbers}, fifth~ed.
	\newblock John Wiley \& Sons, Inc., New York, 1991.
	
	\bibitem{MR815132}
	{\sc N\"{o}bauer, R.}
	\newblock Cryptanalysis of the {R}\'{e}dei-scheme.
	\newblock In {\em Contributions to general algebra, 3 ({V}ienna, 1984)}.
	H\"{o}lder-Pichler-Tempsky, Vienna, 1985, pp.~255--264.
	
	\bibitem{MR3911214}
	{\sc Panario, D., and Reis, L.}
	\newblock The functional graph of linear maps over finite fields and
	applications.
	\newblock {\em Des. Codes Cryptogr. 87}, 2-3 (2019), 437--453.
	
	\bibitem{MR3384830}
	{\sc Qureshi, C., and Panario, D.}
	\newblock R\'{e}dei actions on finite fields and multiplication map in cyclic
	group.
	\newblock {\em SIAM J. Discrete Math. 29}, 3 (2015), 1486--1503.
	
	\bibitem{MR2092636}
	{\sc Rubio, I.~M., and Corrada-Bravo, C.~J.}
	\newblock Cyclic decomposition of permutations of finite fields obtained using
	monomials.
	\newblock In {\em Finite fields and applications}, vol.~2948 of {\em Lecture
		Notes in Comput. Sci.} Springer, Berlin, 2004, pp.~254--261.
	
	\bibitem{MR2436339}
	{\sc Rubio, I.~M., Mullen, G.~L., Corrada, C., and Castro, F.~N.}
	\newblock Dickson permutation polynomials that decompose in cycles of the same
	length.
	\newblock In {\em Finite fields and applications}, vol.~461 of {\em Contemp.
		Math.} Amer. Math. Soc., Providence, RI, 2008, pp.~229--239.
	
	\bibitem{Sakzad2010}
	{\sc {Sakzad}, A., {Panario}, D., {Sadeghi}, M., and {Eshghi}, N.}
	\newblock Self-inverse interleavers based on permutation functions for turbo
	codes.
	\newblock In {\em 2010 48th Annual Allerton Conference on Communication,
		Control, and Computing (Allerton)\/} (2010), pp.~22--28.
	
	\bibitem{MR2966683}
	{\sc Sakzad, A., Sadeghi, M.-R., and Panario, D.}
	\newblock Cycle structure of permutation functions over finite fields and their
	applications.
	\newblock {\em Adv. Math. Commun. 6}, 3 (2012), 347--361.
	
	\bibitem{MR0371855}
	{\sc Shanks, D.}
	\newblock Five number-theoretic algorithms.
	\newblock In {\em Proceedings of the {S}econd {M}anitoba {C}onference on
		{N}umerical {M}athematics ({U}niv. {M}anitoba, {W}innipeg, {M}an., 1972)\/}
	(1973), pp.~51--70. Congressus Numerantium, No. VII.
	
	\bibitem{YTH}
	{\sc Youssef, A.~M., Tavares, S.~E., and Heys, H.~M.}
	\newblock A new class of substitution-permutation networks.
	\newblock {\em Proc. 3rd Annu. Select. Areas Cryptogr. (SAC)\/} (1996),
	132--147.
	
	\bibitem{MR4038906}
	{\sc Zheng, D., Yuan, M., Li, N., Hu, L., and Zeng, X.}
	\newblock Constructions of involutions over finite fields.
	\newblock {\em IEEE Trans. Inform. Theory 65}, 12 (2019), 7876--7883.
	
\end{thebibliography}

\end{document}